\documentclass[10pt]{amsart}
\usepackage{amssymb,amsmath}
\usepackage{subfig}
\usepackage[english]{babel}
\usepackage{graphicx}
\usepackage{tabularx}
\usepackage{enumerate, hyperref}

\newtheorem{theorem}{Theorem}
\newtheorem{lemma}[theorem]{Lemma}

\newtheorem{proposition}[theorem]{Proposition}
\newtheorem{corollary}[theorem]{Corollary}
\newtheorem{remark}[theorem]{Remark}

\author{Pierluigi Vellucci, Alberto Maria Bersani}

\title{The class of Lucas-Lehmer polynomials.}

\begin{document}
\maketitle

\begin{abstract}
In this paper we introduce a new sequence of polynomials, which follow the same recursive rule of the well-known Lucas-Lehmer integer sequence. We show the most important properties of this sequence, relating them to the Chebyshev polynomials of the first and second kind.
\end{abstract}

\section{Introduction}

In \let\thefootnote\relax\footnotetext{2010 Mathematics Subject Classification: 42C05, 33C05.} this paper \let\thefootnote\relax\footnotetext{Key words and phrases: Chebyshev polynomials; Lucas-Lehmer primality test; Lucas-Lehmer primality numbers; orthogonal polynomials; zeros of polynomials.} we study a class of polynomials $L_n(x) = L_{n-1}^2(x) - 2$, which, at the best of our knowledge, are here introduced for the first time, created by means of the same iterative formula used to build the well-known Lucas-Lehmer sequence, employed in primality tests \cite{15:15,16:16,10:10,11:11,9:9}. It is clearly crucial to choose the first term of the polynomial sequence. In this paper we consider $L_{0}=x$.

In this paper we show some properties of these polynomials, in particular discussing the link among the Lucas-Lehmer polynomials and the Chebyshev polynomials of the first and second kind \cite{12:12,13:13,8:8}. The Chebyshev polynomials are well-known and, although they have been known and studied for a long time, continue to play an important role in recent advances in many areas of mathematics such as Algebra, Numerical Analysis, Differential Equations and Number Theory (see, for instance: \cite{babusci2014chebyshev,Bel,Bhr,Cas,dattoli2015cardan,Gul,Mihaila,Swe,Yama,BEJS}) and new other properties of theirs continue to be discovered (\cite{Bel,Cas,dattoli2001note,Dilc}).

In particular, in the spirit of some existing results on the Chebyshev polynomials and the nested square roots (see, for example, \cite{More,WitSl2}), we show that the zeros of the Lucas-Lehmer polynomials can be written in terms of nested radicals.

There are many classes of polynomials which are related to the Chebyshev polynomials, such as \cite{Bhr,Datt,HofWi,Hor,WitSl}. In the spirit of some of these works - if $L_n(x)$, $T_n(x)$, $U_n(x)$ denote (respectively) the nth Lucas-Lehmer polynomials, the Chebyshev polynomials of the first and second kind - we can consider the polynomials $L_n$ as a generalization of the so-called \emph{modified or shifted Chebyshev polynomials}, by introducing an appropriate change of variable $t=f(x)$.

We will now outline the content of this paper. In Sections \ref{sec:1} and \ref{sec:2} we introduce the Lucas-Lehmer polynomials and we show their main properties. Furthermore, we give a recursive formula for the sequence of the first nonnegative zeros of $L_n(x)$, in terms of nested radicals. In Section \ref{sec:3} we show some relations among the Lucas - Lehmer polynomials $L_n(x)$ and the Chebyshev polynomials of the first and second kind, determining several new properties for the former.

In Section \ref{sec:4} we show some generalizations of the Lucas-Lehmer map, having the same properties of $L_n$. In Section \ref{sec:6} we list some further perspectives and developments of the theory.

\section{First iterations of the Lucas-Lehmer map.}
\label{sec:1}

Let us consider the iterative map
\begin{equation}
\label{eq:Lucas_Lehmer}
L_{n}(x) =L_{n-1}(x)^{2}-2 \qquad \qquad ; \qquad \qquad L_0(x) = x \ .
\end{equation}
Assuming $L_{0}=x$ as the initial value, let us construct the first terms of the sequence.
The function $L_{1}(x)=x^{2}-2$ represents a parabola with two zeros $z_{1,2}=\pm \sqrt{2}$ and one minimum point in $(0,-2)$; $L_{2}(x)=(x^{2}-2)^{2}-2= 2 \left(1- 2 x^2 + \frac{x^4}{2} \right)$, shown in Fig. \ref{fig:L_2} contains four zeros: $z_{1\div 4}=\pm \sqrt{2 \pm \sqrt{2}}$. From the derivative of $L_{2}(x)$, $L_{2}'(x)=4x\cdot (x^{2}-2)=4x\cdot L_{1}(x)$ it is possible to determine the critical points of the function: $x_1=0$ (minimum), $x_{2,3}=\pm \sqrt{2}$ (maximum).\
Since $L_{2}(x)= 2 \left(1- 2 x^2 + \frac{x^4}{2} \right) = 2 \cos(2x)+ \textit{o}(x^3)$, for $x\rightarrow 0$ we have $L_{2}(x)\sim 2\cos(2x)$.
\begin{figure}[tb]
\centering
\includegraphics[scale=0.30]{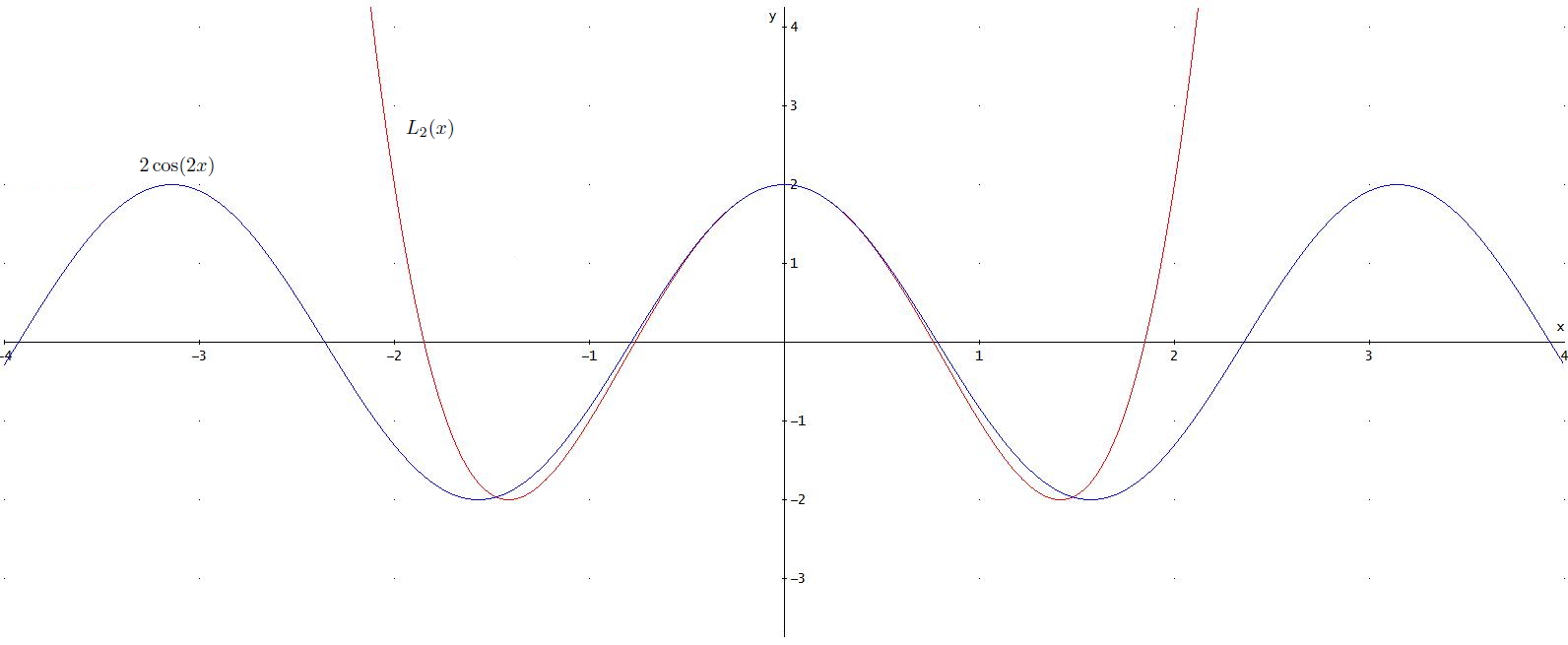}
\caption{comparison between $L_{2}(x)$ and $2\cos(2x)$.}
\label{fig:L_2}
\end{figure}

\begin{figure}[tb]
\centering
\includegraphics[scale=0.30]{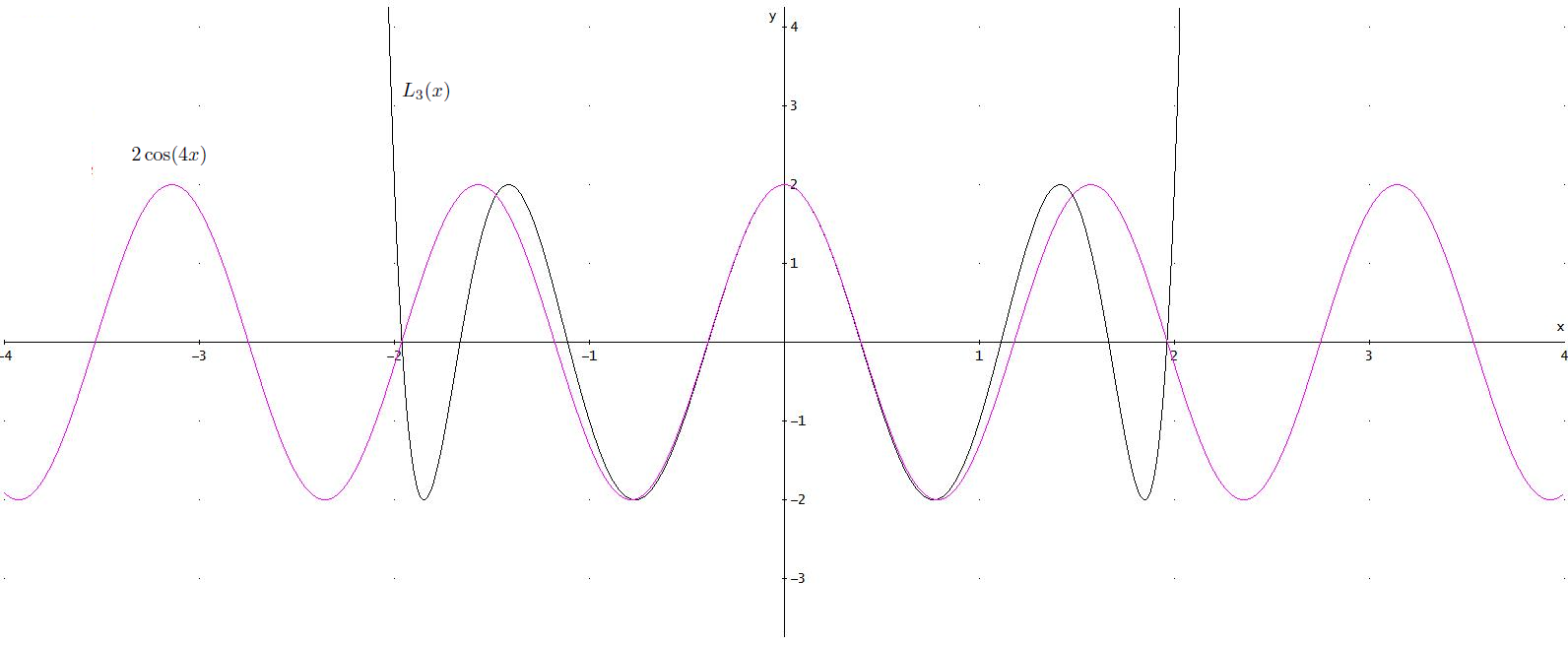}
\caption{comparison between $L_{3}(x)$ and $2\cos(4x)$.}
\label{fig:L_3}
\end{figure}
The zeros of the function $L_{3}(x)=((x^{2}-2)^{2}-2)^{2}-2 = 2 \left(1 - 8x^2 + \textit{o}(x^3)\right)$, whose graph is shown in Fig. \ref{fig:L_3}, are eight: $z_{1\div 8}=\pm \sqrt{2 \pm \sqrt{2 \pm \sqrt{2}}}$. The critical points are: $x_1 = 0$, $x_{2,3}= \pm \sqrt{2}$, $x_{4,5,6,7}=\pm \sqrt{2 \pm \sqrt{2}}$. Besides $L_{3}(x)\sim 2\cos(4x)$ for $x\rightarrow 0$.
\begin{figure}[tb]
\centering
\includegraphics[scale=0.30]{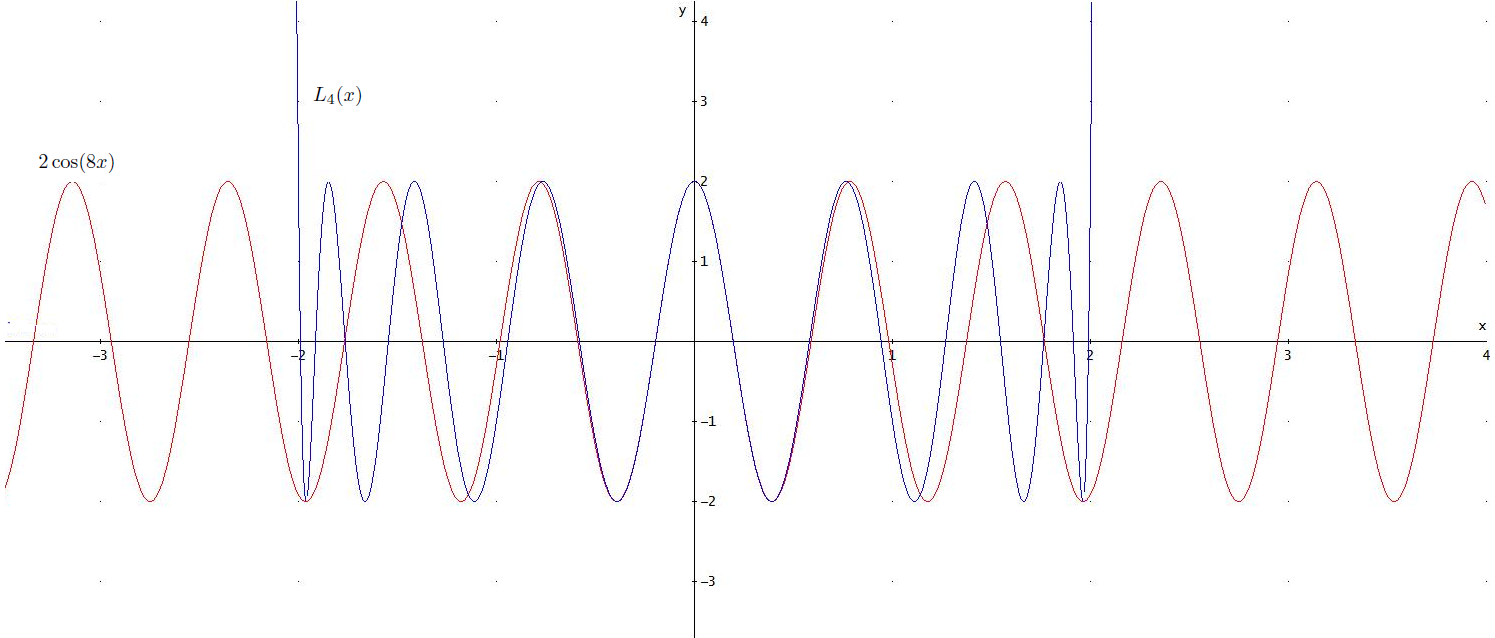}
\caption{comparison between $L_{4}(x)$ and $2\cos(8x)$.}
\label{fig:L_4}
\end{figure}
The zeros of the function (shown in Fig. \ref{fig:L_4}) $L_{4}(x)=(((x^{2}-2)^{2}-2)^{2}-2)^{2}-2$ are sixteen: $z_{1\div 16}=\pm \sqrt{2 \pm \sqrt{2 \pm \sqrt{2 \pm \sqrt{2}}}}$. The critical points follow the same general rule which is possible to guess observing the previous iterations; moreover it results again $L_{4}(x)\sim 2\cos(8x)$ for $x\rightarrow 0$.
It must be noted that: $L_1(\pm \sqrt{2}) = 0$, $L_2(\pm \sqrt{2}) = - 2$, $L_n(\pm \sqrt{2}) = 2 \quad \forall n \ge 3$; $L_0(0) = 0$, $L_1(0) = - 2$, $L_n(0) = 2 \quad \forall n \ge 2$; $L_0(-2) = -2$, $L_n(-2) = 2 \quad \forall n \ge 1$; $L_n(2) = 2 \quad \forall n \ge 0$. Let us observe that the numerical sequence $L_n(\sqrt{6})$ (OEIS, On-Line Encyclopedia of Integer Sequences, \url{http://oeis.org/A003010}) was used, as we said before, in the Lucas-Lehmer primality test ~\cite{9:9, 10:10, 11:11}.

\section{Zeros and critical points.}
\label{sec:2}

Taking into account the considerations of the previous section, we can in general state the following proposition (whose proof is quite simple and is omitted for brevity)
\begin{proposition}
At each iteration the zeros of the map $L_n (n \geq 1)$ have the form
\begin{equation}
\label{eq:prop2a}
\pm  \sqrt{2\pm\sqrt{2\pm\sqrt{2\pm\sqrt{2\pm...\pm\sqrt{2}}}}}
\end{equation}
\end{proposition}

How can we order these zeros? Considering only positive zeros (being $L_n$ a symmetric function), the first sign on the left inside the root must be negative. Let us set $\sqrt{2\pm\sqrt{x_{2}}} =: \sqrt{x_1}$ and $\sqrt{2\pm\sqrt{y_{2}}} =: \sqrt{y_1}$ be two generic roots chosen from those expressed in (\ref{eq:prop2a}). We wonder when $\sqrt{x_1}>\sqrt{y_1}$. Then we have several options.
\begin{description}
  \item[opz1] If $x_{1}=2+\sqrt{x_{2}}, y_{1}=2+\sqrt{y_{2}}$ we have that: $\sqrt{x_1}>\sqrt{y_1}$ $\leftrightarrow$ $2+\sqrt{x_{2}}>2+\sqrt{y_{2}}$ whence $x_{2}>y_{2}$, that is the examination moves to the next step (and we apply again opz1,2,3,4).
  \item[opz2] If $x_{1}=2-\sqrt{x_{2}}, y_{1}=2+\sqrt{y_{2}}$ we have that: $\sqrt{x_1}>\sqrt{y_1}$ $\leftrightarrow$ $2-\sqrt{x_{2}}>2+\sqrt{y_{2}}$, that is impossible, so we have $\sqrt{x_1}<\sqrt{y_1}$.
  \item[opz3] If $x_{1}=2+\sqrt{x_{2}}, y_{1}=2-\sqrt{y_{2}}$ we have that: $\sqrt{x_1}>\sqrt{y_1}$ $\leftrightarrow$ $2+\sqrt{x_{2}}>2-\sqrt{y_{2}}$, always satisfied and we have $\sqrt{x_1}>\sqrt{y_1}$.
  \item[opz4] If $x_{1}=2-\sqrt{x_{2}}, y_{1}=2-\sqrt{y_{2}}$ we have that: $\sqrt{x_1}>\sqrt{y_1}$ $\leftrightarrow$ $2-\sqrt{x_{2}}>2-\sqrt{y_{2}}$, whence $x_{2}<y_{2}$ then we have to check the following step (we apply again opz1,2,3,4). %ma negandone le conclusione a causa del verso invertito della disuguaglianza appena ottenuta).
\end{description}

We show now what we argued in the previous steps.

\begin{theorem}
\label{theo:relazione_fondamentale}
For $n\geq 2$
%e $x\rightarrow 0$
we have
\begin{equation}
\label{eq:theorem_relaz_fondam}
L_{n}(x)= 2\cos(2^{n-1}x) + o(x^{3})
\end{equation}
\end{theorem}
\begin{proof}
Taking into consideration the McLaurin expansion of the cosine, to prove formula \eqref{eq:theorem_relaz_fondam} is equivalent to show that
\begin{equation}
\label{eq:mclaurin}
L_n(x) = 2 - 2^{2n-2}x^2 + o(x^3) = 2 - 4^{n-1}x^2 + o(x^3) \ .
\end{equation}
Let us proceed by means of induction principle.
For $n=2$ we have $L_{2}(x)=(x^{2}-2)^{2}-2=$$x^{4}-4 x^{2} + 2 = 2 - 4 x^2 + o(x^3)$.
Consider then the McLaurin polynomial of the second order of $2\cos(2x)$: it is $\frac{4}{3} x^{4}-4 x^{2} + 2$, which proves the relation for $n=2$. Suppose now as true formula (\ref{eq:theorem_relaz_fondam}) for a generic index $n$ and proceed to check the case $n+1$:
\begin{align}
\label{eq:induz_npiu1_theorem_relaz_fondam}
%&L_{n+1}=L_{n}^{2}-2 = [2\cos(2^{n-1}x) + o(x^{3})]^{2}-2 = \notag \\
%&=\ 4\cos^{2}(2^{n-1}x)+o(x^{6})+4\ \cos(2^{n-1}x) o(x^{3})-2
&L_{n+1}=L_{n}^{2}-2 = [2 - 4^{n-1} x^2 + o(x^{3})]^{2}-2 = \notag \\
&=\ 2 - 4^n x^2 + o(x^{2})
\end{align}
It is also known that McLaurin polynomial of $2\cos(2^{n}x)$ is $2-2^{2n}\cdot x^{2}+R_{3}$. We can therefore conclude that $2\cos(2^{n-1}x)$ and $L_{n}(x)$ have the same coefficients up to the second order, which concludes the proof.
\end{proof}

We are interested in determining the distribution of minima and maxima for each $L_n$. To this aim, now we are going to show an important property of the polynomials $L_n$.
\begin{lemma}
\label{propo:prima}
For each $n\geq 2$ we have
\begin{equation}
\label{eq:derivata}
\frac{d}{dx}L_{n}(x)=2^{n}\ x \ \prod_{i=1}^{n-1}L_{i}(x)
\end{equation}
\end{lemma}
\begin{proof}
Let us proceed by induction. If $n=2$
\begin{equation}
\frac{d}{dx}L_{2}(x)=\frac{d}{dx}[(x^{2}-2)^{2}-2]=4x(x^{2}-2)=4x\ L_{1}(x)
\end{equation}
Now we are going to check it for $n + 1$. For the function \eqref{eq:Lucas_Lehmer}
\begin{equation}
\frac{d}{dx}L_{n+1}(x)=\frac{d}{dx}[L_{n}^{2}(x)]=2 L_{n}(x) \ \frac{d}{dx}L_{n}(x)
\end{equation}
Replacing it with (\ref{eq:derivata})  we will have at the end:
\begin{equation}
\frac{d}{dx}L_{n+1}(x)=2 L_{n} (x) \cdot \left[\ 2^{n}\ x \ \prod_{i=1}^{n-1}L_{i}(x) \right]= 2^{n+1}\ x \ \prod_{i=1}^{n}L_{i}(x)
\end{equation}

\end{proof}

Let $M_{n}$ be the set of the critical points and be $Z_{n}$ the set of the zeros of $L_{n}(x)$; we obtain the following results.
\begin{proposition}
\label{propo:seconda}
For each $n\geq 2$ we have
\begin{equation}
\label{eq:M_L}
M_{n}=M_{n-1}\cup Z_{n-1}=M_{1}\cup \bigcup_{i=1}^{n-1} Z_{i}
\end{equation}
with $card(Z_{n})=2^{n}$.
\end{proposition}
\begin{proof}
Let us first find the critical points of $L_{n}(x)$, imposing $\frac{d}{dx}L_{n}(x)=0$, from which it results
\begin{equation}
2 L_{n-1}(x) \ \frac{d}{dx}L_{n-1}(x)=0
\end{equation}
which vanishes either if $L_{n-1}(x)=0$ (finding the points of $Z_{n-1}$) or if $\frac{d}{dx}L_{n-1}(x)=0$ (determining the points of $M_{n-1}$). Therefore it is proved that $M_{n}=M_{n-1}\cup Z_{n-1}$.
To prove the second equality, it is sufficient to observe that the right hand side of (\ref{eq:derivata}) vanishes either if $x=0$ or if $L_i(x)=0$ for some $i=1, ..., n-1$; thus we obtain the set $\displaystyle \bigcup_{i=1}^{n-1} Z_{i}$,
%\begin{equation}
%L_{1}(x)= ... = L_{n-1}(x)=0 \ \ \Rightarrow \ \ \bigcup_{i=1}^{n-1} Z^{i}
%\end{equation}
which proves the statement.
\end{proof}

\begin{proposition}
\label{propo:secondabis}
For each $n\in N$ we have $card(M_{n})=2^{n}-1$. Furthermore, let $M_{n}^{+}$ be the set of the positive critical points of $L_{n}(x)$; we have that $card(M_{n}^{+})=2^{n-1}-1$.
\end{proposition}
\begin{proof}
We must show that $card(M_{n})=2^{n}-1$; proceeding by induction: if $n=1$, then $L_{1}(x)=x^{2}-2$ is a parabola having only a minimum, at the point $(0,-2)$. Now we are going to check it for $n + 1$, having assumed it true for a generic $n \geq 2$.
From proposition \ref{propo:seconda}, we have, for $n>1$:
\begin{equation}
M_{n+1}=M_{n}\cup Z_{n}\ \Rightarrow \ card(M_{n+1})=card(M_{n})+card(Z_{n})
\end{equation}
(the intersection between $M_{n}$ and $Z_{n}$ being empty). From proposition \ref{propo:seconda}, we have that $card(Z_{n})=2^{n}$; besides, by hypothesis, we know that $card(M_{n})=2^{n}-1$. Then $card(M_{n+1})=2^{n}-1+2^{n}=2^{n+1}-1$. Furthermore, if we don't consider the maximum in the origin, we will have $2^{n}-2$ critical points, half of which are positive. Therefore $card(M_{n}^{+})=2^{n-1}-1$,
\end{proof}

Proposition \ref{propo:prima} is very useful because allows us to obtain some interesting properties for the critical points of $L_n(x)$. We already observed that, for every $n \geq 2$, if $x=0$, then $\left((0-2)^2...\right)^2-2=2$ and the point is a maximum. Moreover, for every natural number $j$ such that $1<j<n-1$ we have that the points $x_0$ such that $L_j(x_0)=0$ are maximum points for $L_n$. Indeed $\left(...(\underbrace{L_j}_{=0}-2)^2...\right)^2-2=2$. Instead, the points $x$ such that $L_{n-1}(x)=0$, being $L_n(x)=\underbrace{L_{n-1}^2(x)}_{=0}-2=-2$, are minimum points for $L_n$. Now, from proposition (\ref{propo:prima}) there aren't other critical points; thus we have shown that the set of maximum points of $L_n(x)$ is: $\displaystyle \bigcup_{i=1}^{n-2} Z_{i}\cup \left\{x=0\right\}$, while the set of minimum points of $L_n(x)$ is $Z_{n-1}$.
\begin{remark}
\label{oss:maxmin}
Minimum points for $L_n(x)$ become maximum points for $L_{n+1}(x)$, maximum points for $L_n(x)$ remain maximum points for $L_{n+1}(x)$. This implies that all the local maxima of every $L_n$ are equal to $2$.
\end{remark}

\begin{corollary}
All zeros and critical points of $L_n$ belong to the interval $(-2, 2)$ \footnote{Because of the symmetry of Lucas-Lehmer polynomials, we will study only positive zeros.}.
\end{corollary}

\section{Relationships between Lucas-Lehmer polynomials and Chebyshev polynomials of the first and second kind, and additional properties.}
\label{sec:3}

As we know ~\cite{8:8, 12:12,13:13}, the \emph{Chebyshev polynomials of the first kind} satisfy the recurrence relation
\begin{displaymath}
\begin{cases}
T_{n}(x)=2xT_{n-1}(x)-T_{n-2}(x) \qquad n \geq 2 \\
T_{0}(x)=1, \ \ T_{1}(x)=x \\
\end{cases}
\end{displaymath}
from which it easily follows that for the $n$-th term:
\begin{equation}
T_{n}(x)=\frac{\left(x -\sqrt{x^2-1} \right)^{n}+\left( x +\sqrt{x^2-1} \right)^{n}}{2}
\end{equation}

This formula is valid in $\mathbb R$ for $|x| \geq 1$; here we assume instead that $T_n$, defined in $\mathbb R$, can take complex values, too.

\begin{proposition}
For each $n\geq1$ we have
\begin{equation}
\label{eq:propcheby1}
L_{n}(x)=2\ T_{2^{n-1}}\left(\frac{x^{2}}{2}-1\right)
\end{equation}
\end{proposition}
\begin{proof}
We must show that
\begin{align}
\label{eq:propcheby1diversa}
&L_{n}(x)=\left(\frac{x^{2}}{2}-1 -\sqrt{\left(\frac{x^{2}}{2}-1\right)^2-1} \right)^{2^{n-1}}+ \notag \\
&+\left( \frac{x^{2}}{2}-1 +\sqrt{\left(\frac{x^{2}}{2}-1\right)^2-1} \right)^{2^{n-1}}
\end{align}

This formula is real for $|x| \geq 2$ and complex for $|x| < 2$ and is
true for $n=1$:
\begin{equation}
L_{1}(x)=x^{2}-2 =
\left[\frac{x^{2}}{2}-1 -\sqrt{\left(\frac{x^{2}}{2}-1\right)^2-1} \right]+\left[ \frac{x^{2}}{2}-1
+\sqrt{\left(\frac{x^{2}}{2}-1\right)^2-1} \right] \ .
\end{equation}

We assume true ($\ref{eq:propcheby1}$) for a natural $n$ and write:
\begin{align}
&L_{n+1}(t(x))=L_{n}^{2}(t(x))-2= \notag \\
&=\left(t -\sqrt{t^2-1} \right)^{2^{n}}+\left( t +\sqrt{t^2-1} \right)^{2^{n}}+\notag \\
%&+2\left[\underbrace{\left(t -\sqrt{t^2-1} \right)\ \left( t +\sqrt{t^2-1} %\right)}_{t^{2}-(t^{2}-1)=1}\right]^{2^{n-1}}-2
&+2\left[\left(t -\sqrt{t^2-1} \right)\ \left( t +\sqrt{t^2-1} \right)\right]^{2^{n-1}}-2
\end{align}
where $\displaystyle t(x) = \frac{x^{2}}{2}-1$. Observing that $\left(t -\sqrt{t^2-1} \right)\ \left( t +\sqrt{t^2-1} \right) = 1$, we lastly obtain
\begin{equation}
\label{validapern}
L_{n+1}(t(x))=\left(t -\sqrt{t^2-1} \right)^{2^{n}}+\left( t +\sqrt{t^2-1} \right)^{2^{n}}
\end{equation}
which concludes the proof.

It is observed that the (\ref{validapern}) is true for a generic function $t(x)$. If $n=1$, instead, the only function that satisfies the (\ref{validapern}) is $\displaystyle t(x) = \frac{x^{2}}{2}-1$.
\end{proof}

\begin{proposition}
\label{cor5bis}
The polynomials $L_n(x)$ are orthogonal with respect to the weight function $\frac{1}{4\sqrt{4-x^2}}$ defined on $x\in[-2,2]$.
\end{proposition}
\begin{proof}
Let us consider Chebyshev polynomials of the first kind; then:
$$\int_{-1}^{1}(1-x^2)^{-1/2}T_n(x) T_m(x) dx=0$$
if $m\neq n$ and $m,n\in\mathbb N$. Using this relationship, we must prove that:
\begin{equation}
\frac{1}{4}\int_{-2}^{2} \frac{1}{\sqrt{4-x^2}} L_n(x) L_m(x) dx=0 \ \ \ m\neq n
\end{equation}
or, by (\ref{eq:propcheby1}):
\begin{equation}
2\int_{-2}^{2} \frac{1}{\sqrt{4-x^2}} T_{2^{n-1}}\left(\frac{x^2}{2}-1\right) T_{2^{m-1}}\left(\frac{x^2}{2}-1\right)dx
\end{equation}
for $m\neq n$ and $m,n\in\mathbb N$. From symmetry of the integrand function, putting $t=\frac{x^2}{2}-1$ and solving the integral we obtain the thesis.
\end{proof}
\begin{corollary}
\label{cor6}
Let $x=2\cos\theta$, then the polynomials $L_{n}(x)$ admit the representation
\begin{equation}
\label{formulaconcos}
L_{n}(2\cos\theta)=2\cos\left(2^{n} \theta\right)
\end{equation}
\end{corollary}
\begin{proof}
Note that, in this case, $|x|\leq 2$. Therefore we need to work with radicals of negative numbers. Substituting $x=2\cos\theta$ in (\ref{eq:propcheby1diversa}) we have:
$$L_{n}(2\cos\theta)=\left(\cos2\theta-\imath \sin2\theta \right)^{2^{n-1}}+\left(\cos2\theta+\imath\sin2\theta\right)^{2^{n-1}}$$
which can be rewritten by applying \emph{Euler's identity}: $$\left(e^{-\imath 2\theta} \right)^{2^{n-1}}+\left(e^{+\imath 2\theta} \right)^{2^{n-1}}=2\cos\left(2^{n} \theta \right)$$
\end{proof}

We resume approximation (\ref{eq:theorem_relaz_fondam}) of $L_{n}(x)$ to prove that locally and for $|x_0| \leq 2$ the function $L_{n}(x)$ behaves like a cosine, while globally, in $[-2, 2]$, it oscillates with shorter and shorter periods in the neighborhoods of the endpoints, by means of the following theorem.
\begin{theorem}
\label{theo:altrarelazione_fondamentale}
{Let $x_0$ a generic maximum point of $L_{n}(x)$. For $n\geq 2$ we have
\begin{equation}
\label{eq:altrotheorem_relaz_fondam}
L_{n}(x)= 2\cos(2^{n-1} k (x-x_{0})) + o((x-x_{0})^{2})
\end{equation}
where $k$ is such that $|k|\geq 1$ and is increasing with $x_0$, for fixed $n$.}
\end{theorem}
\begin{proof}
For $n=2$ it is sufficient recall Theorem \ref{theo:relazione_fondamentale}. In this case $k=1$. Let us now suppose the claim to be true for some natural $n$ and proceed by induction for $n+1$:
\begin{align}
\label{eq:induz_npiu1_altrotheorem_relaz_fondam}
L_{n+1}(x) = &L_{n}^{2}(x) -2 = [2\cos(2^{n-1}k(x-x_{0})) + o((x-x_{0})^{2})]^{2}-2 = \notag \\
&=\ 4\cos^{2}[2^{n-1}k(x-x_{0})]+o[(x-x_{0})^{4}]+ \notag \\
&+4\ \cos[2^{n-1}k(x-x_{0})]\ o[(x-x_{0})^{2}]-2
\end{align}
from which, by means of well known trigonometric formulas, we arrive to $L_{n+1}(x) = L_n^2 - 2 = 2\cos(2^{n}k(x-x_{0})) + o((x-x_{0})^{2})$ if $x\rightarrow x_{0}$. From Remark (\ref{oss:maxmin}), the point $x_0$ is a maximum point for $L_n(x)$ and $L_{n+1}(x)$. Now we aim to prove that $|k| \geq1$. The second-order Taylor expansion of the right hand side of (\ref{eq:altrotheorem_relaz_fondam}), centered in $x_0$, is
\begin{equation}
\label{eq:taylor_coseno_polo_xo}
2-2^{2(n-1)}k^{2}(x-x_{0})^{2}+ o((x-x_{0})^{2})
\end{equation}
For what concerns the left hand side of (\ref{eq:altrotheorem_relaz_fondam}), we observe that
$L_{n}(x_{0})=2$, being $x_{0}$ a maximum point. Let us observe that equation (\ref{eq:propcheby1diversa})
\begin{align}
\label{ln+ln-}
&L_{n}(x)=\Biggl(\frac{x^{2}}{2}-1+\sqrt{\Bigl(\frac{x^{2}}{2}-1\Bigr)^{2}-1}\ \Biggr)^{2^{n-1}}+ \notag \\
&+\Biggl(\frac{x^{2}}{2}-1-\sqrt{\Bigl(\frac{x^{2}}{2}-1\Bigr)^{2}-1}\ \Biggr)^{2^{n-1}}=L_{n}^{+}(x)+L_{n}^{-}(x)
\end{align}
must be understood with values in the complex field, because, due to
$$\underbrace{\sqrt{2+\sqrt{2+\sqrt{2+...+\sqrt{2}}}}}_{n}=2\ \cos\Bigl(\frac{\pi}{2^{n+1}}\Bigr) < 2$$
all the critical points have absolute value less or equal to 2. Then the derivative of $L_{n}(x)$ is
\begin{equation}
L_{n}'(x)=\frac{d}{dx}\left( L_{n}^{+}(x)+L_{n}^{-}(x) \right)
\end{equation}
with
$$\frac{d}{dx}L_{n}^{+}(x)=2^{n-1}\ \frac{x\ L_{n}^{+}(x)}{\sqrt{\left(\frac{x^{2}}{2}-1\right)^{2}-1}}$$
and
$$\frac{d}{dx}L_{n}^{-}(x)=-2^{n-1}\ \frac{x\ L_{n}^{-}(x)}{\sqrt{\left(\frac{x^{2}}{2}-1\right)^{2}-1}}$$
whence
\begin{equation}
 \ L_{n}'(x)=2^{n-1}\ \frac{x}{\sqrt{\left(\frac{x^{2}}{2}-1\right)^{2}-1}} \left[ L_{n}^{+}(x)-L_{n}^{-}(x) \right]
 =\frac{2^{n}}{\sqrt{x^{2}-4}}\left[ L_{n}^{+}(x)-L_{n}^{-}(x) \right]
\end{equation}
which must vanish when calculated in $x=x_{0}$, maximum point. For the sake of simplicity, let us consider only $x > 0$. The second order derivative is
\begin{equation}
L_{n}''(x)=\frac{2^{n}\left\{(x^{2}-4) \Bigl[\frac{d}{dx}L_{n}^{+}(x)- \frac{d}{dx}L_{n}^{-}(x)\Bigr] - x \left( L_{n}^{+}(x)-L_{n}^{-}(x) \right) \right\}}{(x^{2}-4)\sqrt{x^{2}-4}}
\end{equation}
which can be rewritten as
\begin{align}
&L_{n}''(x)=2^{n}\left[ \frac{\frac{d}{dx}\left(L_{n}^{+}(x)-L_{n}^{-}(x)\right)}{\sqrt{x^{2}-4}} - \frac{x L_{n}'(x)}{2^{n}(x^{2}-4)} \right]= \notag \\
&=2^{n}\left[ \frac{2^{n}\left(L_{n}^{+}(x)+L_{n}^{-}(x)\right)}{(\sqrt{x^{2}-4})^{2}} - \frac{x L_{n}'(x)}{2^{n}(x^{2}-4)} \right] \ .
\end{align}
\noindent
We calculate it in $x=x_0$:
\begin{equation}
L_{n}''(x_{0})=2^{n}\left[ \frac{2^{n}L_{n}(x_{0})}{x_{0}^{2}-4} - \frac{x_{0} L_{n}'(x_{0})}{2^{n}(x_{0}^{2}-4)} \right]= \frac{2^{2n}L_{n}(x_{0})}{x_{0}^{2}-4}= \frac{2^{2n+1}}{x_{0}^{2}-4}
\end{equation}
since $L_{n}'(x_{0})=0$ and $L_{n}(x_{0})=2$. Thus we have the Taylor expansion
\begin{equation}
L_{n}(x)=2+\frac{2^{2n}}{x_{0}^{2}-4} (x-x_{0})^{2}+o((x-x_{0})^{2}) \ .
\end{equation}
Equating it to (\ref{eq:taylor_coseno_polo_xo}) gives
\begin{equation}
\frac{2^{2n}}{x_{0}^{2}-4}=-2^{2(n-1)}k^{2} \ \Rightarrow \ \frac{4}{4-x_{0}^{2}}=k^{2} \ \Rightarrow \ k=\pm \frac{1}{\sqrt{1-x_{0}^{2}/4}}
\end{equation}
It is easy to verify that $k$ is such that $|k|\geq 1$ and increasing with $x_0>0$.
\end{proof}

As those of the first kind, the \emph{Chebyshev polynomial of the second kind} are defined by a recurrence relation ~\cite{8:8, 12:12, 13:13}:
\begin{displaymath}
\begin{cases}
U_{0}(x)=1, \ \ U_{1}(x)=2x \\
U_{n}(x)=2xU_{n-1}(x)-U_{n-2}(x) \quad \forall n \geq 2\\
\end{cases}
\end{displaymath}
which is satisfied by
\begin{equation}
\label{eq:Usemplice}
U_{n}(x)=\sum_{k=0}^n (x + \sqrt{x^2 -1})^k (x - \sqrt{x^2-1})^{n-k} \quad \forall x \in [-1, 1] \ .
\end{equation}

This relation is equivalent to
\begin{equation}
\label{eq:Ufrazione}
U_{n}(x)=\frac{\left(x +\sqrt{x^2-1} \right)^{n+1}-\left( x -\sqrt{x^2-1} \right)^{n+1}}{2\sqrt{x^2-1}}
\end{equation}
where the radicals assume real values for each $x \in (-1, 1)$. From continuity of function (\ref{eq:Usemplice}), we observe that (\ref{eq:Ufrazione}) can be extended by continuity in $x = \pm 1$, too. It can therefore be put $U_n(\pm 1) = (\pm 1)^n (n+1)$ in (\ref{eq:Ufrazione}).

\begin{proposition}
\label{con_sommatoria}
For each $n\geq1$ we have
\begin{equation}
\label{eq:propcheby2}
\prod_{i=1}^{n}L_{i}(x)=U_{2^{n}-1}\left(\frac{x^{2}}{2}-1\right)
\end{equation}
\end{proposition}
\begin{proof}
Also in this case, the formulas are defined on complex numbers. By (\ref{eq:Ufrazione}) we must demonstrate that:
\begin{align}
\label{eq:propcheby2diversa}
&\prod_{i=1}^{n}L_{i}(x)=\frac{\left(\frac{x^{2}}{2}-1 +\sqrt{\left(\frac{x^{2}}{2}-1\right)^2-1} \right)^{2^{n}}}{2\sqrt{\left(\frac{x^{2}}{2}-1\right)^2-1}}- \notag \\
&+\frac{\left( \frac{x^{2}}{2}-1 -\sqrt{\left(\frac{x^{2}}{2}-1\right)^2-1} \right)^{2^{n}}}{2\sqrt{\left(\frac{x^{2}}{2}-1\right)^2-1}} \ .
\end{align}

We proceed by induction on $n$. First of all, let us observe that when $n=1$ we have $L_1(x) = x^2 - 2 = U_1\left(\frac{x^{2}}{2}-1\right)$.

For the inductive step, let $n>1$ be an integer, and assume that the proposition holds for $n$;
by multiplying both sides of (\ref{eq:propcheby2}) by $L_{n+1}(x)$ we obtain:
\begin{equation}
\prod_{i=1}^{n+1}L_{i}(x)=U_{2^{n}-1}\left(\frac{x^{2}}{2}-1\right)\ L_{n+1}(x)
\end{equation}
Thus, the proposition holds for $n+1$ if
\begin{equation}
U_{2^{n+1}-1}\left(\frac{x^{2}}{2}-1\right)=U_{2^{n}-1}\left(\frac{x^{2}}{2}-1\right)\ L_{n+1}(x)
\end{equation}
Let's focus on the right hand side, setting $\displaystyle t = \frac{x^2}{2}-1$:
\begin{align}
&= \underbrace{\sum_{k=0}^{2^{n}-1} (t + \sqrt{t^2 -1})^{k+2^n} (t - \sqrt{t^2-1})^{2^{n}-1-k}}_B+ \notag \\
&+\underbrace{\sum_{k=0}^{2^{n}-1} (t + \sqrt{t^2 -1})^k (t - \sqrt{t^2-1})^{2^{n+1}-1-k}}_A
\end{align}
where
\begin{align}
&A=\sum_{k=0}^{2^{n+1}-1} (t + \sqrt{t^2 -1})^k (t - \sqrt{t^2-1})^{2^{n+1}-1-k}+ \notag \\
&-\sum_{k=2^n}^{2^{n+1}-1} (t + \sqrt{t^2 -1})^k (t - \sqrt{t^2-1})^{2^{n+1}-1-k} \notag \\
&B=\sum_{k=0}^{2^{n}-1} (t + \sqrt{t^2 -1})^{k+2^n} (t - \sqrt{t^2-1})^{2^{n}-1-k} \notag \\
&=\sum_{j=2^n}^{2^{n+1}-1} (t + \sqrt{t^2 -1})^j (t - \sqrt{t^2-1})^{2^{n+1}-1-j}
\end{align}
therefore $A+B$ is just equal to $U_{2^{n+1}-1}\left(t\right)$,  and this completes the proof.
\end{proof}
After having calculated $L_n(2\cos\theta)$ in (\ref{formulaconcos}), now let us calculate $U_{2^{n}-1}(x^2/2 -1)$ for $x=2\cos\theta$ by (\ref{eq:propcheby2diversa}).
$$\prod_{i=1}^{n}L_{i}(2\cos\theta)=\frac{\left(\cos2\theta+\imath\sin2\theta \right)^{2^{n}}-\left( \cos2\theta - \imath\sin2\theta \right)^{2^{n}}}{2\imath\sin2\theta}$$
from which and \emph{Euler identity}, we get
\begin{equation}
\label{euler}
\prod_{i=1}^{n}L_{i}(2\cos\theta)= \frac{\left(e^{+\imath 2\theta} \right)^{2^{n}}-\left(e^{-\imath 2\theta} \right)^{2^{n}}}{2\imath\sin2\theta}=\frac{\sin\left(2^{n+1} \theta \right)}{\sin2\theta} \ .
\end{equation}

For $|x| \leq 2$ we can show another formula for $L_n$. Let us come back to (\ref{eq:propcheby1diversa}):
\begin{align}
&L_{n}(x)=\left(\frac{x^{2}}{2}-1 -\sqrt{\left(\frac{x^{2}}{2}-1\right)^2-1} \right)^{2^{n-1}}+ \notag \\
&+\left( \frac{x^{2}}{2}-1 +\sqrt{\left(\frac{x^{2}}{2}-1\right)^2-1} \right)^{2^{n-1}}
\end{align}

In this case $|x| \leq 2$; we change the sign inside the radical, factorizing out the imaginary unit:
\begin{align}
&L_{n}(x)=\left(\frac{x^{2}}{2}-1 -\imath\sqrt{1-\left(\frac{x^{2}}{2}-1\right)^2} \right)^{2^{n-1}}+ \notag \\
&+\left( \frac{x^{2}}{2}-1 +\imath\sqrt{1-\left(\frac{x^{2}}{2}-1\right)^2} \right)^{2^{n-1}}
\end{align}
We then calculate the powers of two complex conjugate numbers $L_n^+$ and $L_n^-$, depending on the variable $x$. With the notation introduced in (\ref{ln+ln-}), the absolute value of both complex numbers is unitary, since
\begin{equation}
|L_n^+| = |L_n^-| = \sqrt{\left(\frac{x^{2}}{2}-1\right)^2 + 1-\left(\frac{x^{2}}{2}-1\right)^2} = 1 \ .
\end{equation}

Moreover, since $L_1(\pm \sqrt{2}) = 0 \ ; \ L_2(\pm \sqrt{2}) = - 2 \ ; \  L_n(\pm \sqrt{2}) = 2 \quad \forall n \ge 3 \ $, then the argument of $L_n(\pm \sqrt{2})$ is $0$ for every $n \ge 3$. In the other cases, since, when $|x| \leq 2$, we can write $x = 2 \cos(\vartheta)$, thus $\displaystyle \frac{x^2}{2} - 1 = \cos(2 \vartheta)$; thus for $|x| \neq \sqrt{2}$ we can also put
\begin{equation}
\label{theta}
\vartheta(x)= \frac{1}{2}\arctan\left[\frac{\sqrt{1-\left(\frac{x^{2}}{2}-1\right)^2}}{\frac{x^{2}}{2}-1} \right]+b\pi
\end{equation}
where $b$ is a binary digit; thus, using (\ref{formulaconcos}), we obtain $L_{n}(x)=2\cos\left(2^{n}\vartheta(x)\right)$.

By setting further
\begin{equation}
\label{theta_value}
\theta(x)= \frac{1}{2}\arctan\left[\frac{\sqrt{1-\left(\frac{x^{2}}{2}-1\right)^2}}{\frac{x^{2}}{2}-1} \right]
\end{equation}
we can write:
\begin{equation}
\label{Ln_value}
L_{n}(x)=2\cos\left(2^{n}\theta(x)+2^{n}b\pi\right)=2\cos\left(2^{n}\theta(x)\right) \ .
\end{equation}
On the other hand, for very large $|x|$, considering the iterative structure of the map $L_n$, we deduce immediately the asymptotic formula $L_{n}(x) \sim (x^2-2)^{2^{n-1}}$.

\section{$M^a_{n}=2a \left(M^a_{n-1}\right)^{2}-\frac{1}{a}$ map.}
\label{sec:4}

The considerations made in the previous sections on the map $L_n$ can be extended to an entire class of maps, obtained through the iterated formula $M^a_{n}=2a \left(M^a_{n-1}\right)^{2}-\frac{1}{a} \ , \ a >0$, with $M^a_0(x) = x$. It follows that

\begin{equation}
\displaystyle M^a_0(x) = x \quad ; \quad M^a_1(x) = 2a x^2 - \frac{1}{a} \quad ; \quad M^a_2(x) = 8a^3x^4 - 8ax^2 + \frac{1}{a} \quad ...
\end{equation}

Note that the map $L_n$ is a particular case of $M^a_n$, obtained by setting $a=1/2$. We briefly show that the map $M^a_n$ satisfies similar properties as those proven for $L_n$.

\begin{proposition}
For $n\geq2$ we have
\begin{equation}
\label{eq:prop1}
M^a_{n}(x)= \frac{1}{a} \cdot \cos(a\ 2^{n}x)+o(x^{2})
\end{equation}
\end{proposition}
\begin{proof}
We must show that:
\begin{equation}
M^a_{n}(x)= \frac{1}{a} - a 2^{2n-1}x^{2}+o(x^{2})
\end{equation}
where we take into account the McLaurin polynomial of cosine. We proceed by induction.
For $n=2$:
\begin{equation}
M^a_{2}(x)= 2a \left(2a x^{2}-\frac{1}{a}\right)^{2}-\frac{1}{a}= \frac{1}{a} - 8ax^{2}+o(x^{2})
\end{equation}
Let us consider the second order McLaurin polynomial of $\frac{1}{a} \cdot \cos(4ax)$: it is just $\frac{1}{a} - 8ax^{2}+o(x^{2})$, thus verifying the relation for $n=2$. Let us now assume (\ref{eq:prop1}) is true for a generic $n$, and deduce that it is also true for $n+1$:
\begin{align}
&M^a_{n+1}=2a \left(M^a_{n}\right)^{2}-\frac{1}{a}=2a\left[\frac{1}{a} - a 2^{2n-1}x^{2}+o(x^{2}) \right]^{2}-\frac{1}{a}=\notag \\
&=\frac{1}{a} - a 2^{2n+1}x^{2}+o(x^{2})
\end{align}
which is in fact the McLaurin polynomial of $\frac{1}{a} \cdot \cos(a\ 2^{n+1}x)$.
\end{proof}

\begin{proposition}
At each iteration the zeros of the map $M^a_n (n \geq 1)$ have the form
\begin{equation}
\label{eq:prop2}
\pm \frac{1}{2a}\cdot \sqrt{2\pm\sqrt{2\pm\sqrt{2\pm\sqrt{2\pm...\pm\sqrt{2}}}}}
\end{equation}
\end{proposition}
\begin{proof}
It is obvious that at $n=1$ this statement is valid. Now assume that the (\ref{eq:prop2}) is valid for $n$. We have to prove that it is valid for $n+1$:
\begin{equation}
\label{eq:dim_prop2_next}
x^{2}=\frac{1}{2a^{2}} \pm \frac{1}{4a^{2}}\cdot \sqrt{2\pm\sqrt{2\pm\sqrt{2\pm\sqrt{2\pm...\pm\sqrt{2}}}}}
\end{equation}
and placing under the radical sign
\begin{equation}
\label{eq:dim_prop2_next_next}
x=\pm \sqrt{\frac{1}{2a^{2}} \pm \frac{1}{4a^{2}}\cdot \sqrt{2\pm\sqrt{2\pm\sqrt{2\pm\sqrt{2\pm...\pm\sqrt{2}}}}}}
\end{equation}
the thesis is obtained.
\end{proof}

\begin{proposition}
For each $n\geq1$ we have
\begin{equation}
\label{eq:propchebyM1}
M^a_{n}(x)=\frac{1}{a}\ T_{2^{n-1}}\left(2a^{2}x^{2}-1\right)
\end{equation}
\end{proposition}
\begin{proof}
We must show that
\begin{align}
\label{eq:propchebyM1diversa}
&M^a_{n}(x)=\frac{\left(2a^{2}x^{2}-1 -\sqrt{\left(2a^{2}x^{2}-1\right)^2-1} \right)^{2^{n-1}}}{2a}+ \notag \\
&+\frac{\left( 2a^{2}x^{2}-1 +\sqrt{\left(2a^{2}x^{2}-1\right)^2-1} \right)^{2^{n-1}}}{2a}
\end{align}
This is verified for $n=1$:
\begin{equation}
M^a_{1}(t(x))=\frac{t+\sqrt{t^{2}-1}+t-\sqrt{t^{2}-1}}{2a}=\frac{2t}{2a}=2a x^{2}-\frac{1}{a}
\end{equation}
where $t=2a^{2}x^{2}-1$. By assumption, we suppose ($\ref{eq:propchebyM1}$) true for $n$ and by $M^a_{n+1}(t(x))=2a(M^a_{n})^{2}(t(x))-\frac{1}{a}$; we get finally the thesis for $n+1$:
\begin{equation}
\label{eq:Mnpiu1}
M^a_{n+1}(t(x))=\frac{\left(t -\sqrt{t^2-1} \right)^{2^{n}}}{2a}+\frac{\left( t +\sqrt{t^2-1} \right)^{2^{n}}}{2a}
\end{equation}
\end{proof}

\begin{remark}
For $\displaystyle |x| \leq \frac{1}{a}$, substituting $x=\frac{1}{a}\cos\theta$ in (\ref{eq:propchebyM1diversa}) we obtain:
\begin{equation}
M^a_{n}\left(\frac{1}{a}\cos\theta\right)=\frac{1}{a}\cos\left(2^{n}\theta\right) \ .
\end{equation}
\end{remark}

\begin{proposition}
For each $n\geq1$ we have
\begin{equation}
\label{eq:propchebyM2}
\prod_{i=1}^{n}M^a_{i}(x)=\left(\frac{1}{2a}\right)^{n}U_{2^{n}-1}\left(2a^{2}x^{2}-1\right)
\end{equation}
\end{proposition}
\begin{proof}
We must first show that the formula is true for $n=1$:
\begin{equation}
\frac{1}{2a}\ U_{1}(t)=\frac{\left(t+\sqrt{t^{2}-1}\right)^{2}-\left(t-\sqrt{t^{2}-1}\right)^{2}}{4a\sqrt{t^{2}-1}}=\frac{t}{a}
\end{equation}
which is true because $\frac{t}{a}=2a x^{2}-\frac{1}{a}=M^a_{1}(x)$. We assume, then, that formula (\ref{eq:propchebyM2}) is true for $n$. We must now show that it is also true for $n+1$. To this aim, let us multiply both sides of (\ref{eq:propchebyM2}) by $M_{n+1}(t)$, expressed in (\ref{eq:Mnpiu1}); the right-hand side becomes
\begin{align}
&\left(\frac{1}{2a}\right)^{n+1}\ \frac{\left(t +\sqrt{t^2-1} \right)^{2^{n+1}}-\left( t -\sqrt{t^2-1} \right)^{2^{n+1}}}{2\sqrt{t^2-1}}= \notag \\
&=\left(\frac{1}{2a}\right)^{n+1}\ U_{2^{n+1}-1}\left(2a^{2}x^{2}-1\right)
\end{align}
\end{proof}

\begin{proposition}
\label{propo:primaMn}
For each $n\geq 2$ we have
\begin{equation}
\label{eq:derivataMn}
\frac{d}{dx}M^a_{n}(x)=(4a)^{n}\ x \ \prod_{i=1}^{n-1}M^a_{i}(x)
\end{equation}
\end{proposition}
\begin{proof}
We have first to prove it is true for $n=2$:
\begin{equation}
\frac{d}{dx}M^a_{2}(x)=\frac{d}{dx}\ \left[2a\left(2ax^2-\frac{1}{a}\right)^{2}-\frac{1}{a}\right]=4a x M^a_{1}(x)
\end{equation}
Assume it is true for $n$ and deduce that (\ref{eq:derivataMn}) is true for $n+1$, too. In fact $M^a_{n+1}=2a (M^a_{n})^{2}-\frac{1}{a}$. Write
\begin{equation}
\frac{d}{dx}M^a_{n+1}(x)=2a\frac{d}{dx}(M^a_{n})^{2}=4a M^a_{n} \ \frac{d}{dx}M^a_{n}
\end{equation}
and using (\ref{eq:derivataMn}) we arrive to:
\begin{equation}
\frac{d}{dx}M^a_{n+1}(x)=4a M^a_{n} \cdot \left[\ (4a)^{n}\ x \ \prod_{i=1}^{n-1}M^a_{i}(x) \right]= (4a)^{n+1}\ x \ \prod_{i=1}^{n}M^a_{i}(x)
\end{equation}
\end{proof}

\begin{remark}
It can be easily shown that, when $\displaystyle |x| \leq \frac{1}{a}$, replacing $x=\frac{1}{a}\cos\theta$ in the expression of $U_{2^{n}-1}\left(2a^{2}x^{2}-1\right)$ and taking into account that $2a^2x^2-1 = \cos(2 \theta)$:
\begin{equation}
\frac{\left(2a^{2}x^{2}-1+\sqrt{(2a^{2}x^{2}-1)^{2}-1}\right)^{2^{n}}-\left(2a^{2}x^{2}-1-\sqrt{(2a^{2}x^{2}-1)^{2}-1}\right)^{2^{n}}}{2\sqrt{(2a^{2}x^{2}-1)^{2}-1}}
\end{equation}
we again get the trigonometric expression (\ref{euler}).
\end{remark}

Factorizing out the minus sign in (\ref{eq:propchebyM1diversa}) and carrying out the imaginary unit from radical, we obtain:
$$M^a_{n}(x)= \frac{1}{2a}\left[(M^{a,+}_n)^{2^{n-1}} + (M^{a,-}_n)^{2^{n-1}}\right]$$
The module of both complex numbers $M^{a,+}_n$ and $M^{a,-}_n$ is unitary; in fact:
\begin{equation}
|M_{n}^{a,+}(x)|=|M_{n}^{a,-}(x)|=\sqrt{\left(2a^{2}x^{2}-1\right)^{2}+1-\left(2a^{2}x^{2}-1\right)^{2}}=1
\end{equation}

Then
\begin{align}
M^a_{n}(x)&=\frac{e^{i 2^n \vartheta(x)}+e^{-i 2^n \vartheta(x)}}{2a}=\frac{1}{a}\cos\left(2^n \vartheta(x) \right) \notag \\
\vartheta(x)&=\frac{1}{2} \arctan\left[\frac{\sqrt{1-\left(2a^{2}x^{2}-1\right)^2}}{2a^{2}x^{2}-1} \right]+b\pi=\theta(x)+b\pi
\end{align}
with $b$ a binary digit, and
\begin{equation}
M^a_{n}(x)=\frac{1}{a}\cos\left(2^n\theta(x)+2^n b\pi\right)=\frac{1}{a}\cos\left(2^n\theta(x)\right)
\end{equation}
If $x=\pm\frac{\sqrt{2}}{2a}$: $M^a_1(\pm\frac{\sqrt{2}}{2a}) = \frac{1}{a}\cos\left(\frac{\pi}{2}\right) = 0$; $M^a_2 (\pm\frac{\sqrt{2}}{2a})=\frac{1}{a}\cos\left(\pi\right)=- \frac{1}{a}$; $M^a_{n}(\pm\frac{\sqrt{2}}{2a})=\frac{1}{a}\cos\left(2^{n-2}\pi\right)=\frac{1}{a} \ ; \ n \geq 3$. Then the argument of $\displaystyle M^a_{n} \left(\pm\frac{\sqrt{2}}{2a}\right)$ is $0$ for every $n \ge 3$. For very large $|x|$, considering the iterative structure of the map $M^a_{n}$, we deduce immediately the asymptotic formula: $$M^a_{n} \sim (2a)^{2^{n-1}-1} \left(2a x^2-\frac{1}{a}\right)^{2^{n-1}},$$ $\forall n \geq 1$.

\section{Conclusions and perspectives.}
\label{sec:6}

In this paper we introduced a class of polynomials which follow the same recursive formula as the Lucas-Lehmer numbers. We showed several properties of the polynomials, including important links with the Chebyshev polynomials, proving their orthogonality with respect to a suitable weight.

This paper intended just to introduce this new class of polynomials. Much more aspects need to be deepened, concerning the properties of the polynomials and their applications.

A further progress in our work will consist in studying the distribution of the zeros of $L_n$. This topic is the subject of a paper in preparation. By Section \ref{sec:3}, the zeros, expressed in terms of nested radicals, allow us to generalize in infinite ways a well-known formula for the approximation of $\pi$. This formula allows us to obtain several other notable mathematical constants as limits of suitably weighted sequences of zeros of the Lucas-Lehmer polynomials $L_n$ and $M^a_n$.

Thanks to their strict link with the Chebyshev polynomials, we could determine other properties of the Lucas-Lehmer polynomials, mainly of integral and asymptotic type. These topics will be subject of future studies. Moreover, it would be interesting to determine and study different classes of Lucas-Lehmer polynomials, for example modifying suitably the first term of the sequence.

Finally, it is well known that Chebyshev polynomials can be applied in several fields of Mathematics, for example in Numerical Analysis and Combinatorics (see, for instance: \cite{Bel,Bhr,Cas,Gul,Mihaila,Swe,Yama,BEJS} and references therein). We think that also the Lucas-Lehmer polynomials could be adapted to solve similar problems.

%\section*{Acknowledgments} The Authors express their sincere appreciation to anonymous referees for offering the possibility of extending and improving the originally submitted note.


\begin{thebibliography}{9}

\bibitem{babusci2014chebyshev}
D. Babusci, G. Dattoli, E. Di Palma, E. Sabia, Chebyshev polynomials and generalized complex numbers, Adv. Appl. Clifford Al., 24(1), 1-10, 2014.
\bibitem{12:12}
H. Bateman,  Higher Trascendental Functions - Vol. II, McGraw Hill, New York, 1953.
\bibitem{Bel}
G. Belforte, P. Gay, G. Monegato, Some new properties of Chebyshev polynomials, J. Comp. Appl. Math., 117, 175-181, 2000.
\bibitem{BEJS}
A.T. Benjamin, L. Ericksen, P. Jayawant, M. Shattuck, Combinatorial Trigonometry with Chebyshev Polynomials, J. Stat. Plan. Inference, 140(8), 2157-2160, 2010.
\bibitem{Bhr}
A. Bhraway,  M.M. Tharwat,  A. Yildirim, A new formula for fractional integrals of Chebyshev polynomials: application for solving multi-term fractional differential equations, Appl. Math. Modell., 37(6), 4245-4252, 2013.
\bibitem{11:11}
D.M. Bressoud, Factorization and Primality Testing, Springer-Verlag, New York, 1989.
\bibitem{Cas}
K. Castillo, R.L. Lambl\'{e}m, A. Sri Ranga, On a moment problem associated with Chebyshev polynomials, Appl. Math. Comp., 218, 9571-9574, 2012.
\bibitem{dattoli2001note}
G. Dattoli, D. Sacchetti, C. Cesarano, A note on Chebyshev polynomials, Ann. Univ. Ferrara, 47(1), 107-115, 2001.
\bibitem{Datt}
G. Dattoli, Integral transforms and Chebyshev-like polynomials, Appl. Math. Comp., 148, 225-234, 2004.
\bibitem{dattoli2015cardan}
D. Dattoli, E. Di Palma, E. Sabia, Cardan Polynomials, Chebyshev Exponents, Ultra-Radicals and Generalized Imaginary Units, Adv. Appl. Clifford Al., 25(1), 81-94, 2015.
\bibitem{Dilc}
K. Dilcher, K.B. Stolarsky, Nonlinear recurrences related to Chebyshev polynomials, Ramanujan J., 1-23, 2014.
\bibitem{13:13}
L. Gatteschi, Funzioni Speciali, UTET, Torino, 1973.
\bibitem{Gul}
M. G{\"{u}}lsu, Y. {\"{O}}zt{\"{u}}rk, M. Sezer, On the solution of the Abel equation of the second kind by the shifted Chebyshev polynomials, Appl. Math. Comp., 217, 4827-4833, 2011.
\bibitem{HofWi}
M.E. Hoffman, W.D. Withers, Generalized Chebyshev polynomials associated with affine Weyl groups, Trans. Amer. Math. Soc., 308(1), 91-104, 1988.
\bibitem{Hor}
A. Horadam, Vieta polynomials, Trans. Amer. Math. Soc., 40, 223-232, 2002.
\bibitem{9:9}
T. Koshy, Fibonacci and Lucas numbers with applications, John Wiley and Sons, New York, 2001.
\bibitem{16:16}
D.H. Lehmer, An Extended Theory of Lucas Functions, Ann. Math., 31, 419-448, 1930.
\bibitem{15:15}
E. Lucas, Th\'{e}orie des Fonctions Numeriques Simplement P\'{e}riodiques, Am. J. Math., 1, 184–196, 1878.
\bibitem{Mihaila}
B. Mihaila, I. Mihaila, Numerical Approximations Using Chebyshev Polynomial Expansions: El-gendi's Method Revisited, J. Phys. A: Math. Gen., 35, 731-746, 2002.
\bibitem{More}
S.G. Moreno, Esther M. Garcia-Caballero, Chebyshev polynomials and nested square roots, J. Math. Anal. Appl., 394, 61-73, 2012.
\bibitem{10:10}
P. Ribenboim, The Book of Prime Number Records, Springer-Verlag, New York, 1988.
\bibitem{8:8}
T.J. Rivlin, Chebyshev Polynomials - 2nd Edition, John Wiley and Sons, New York, 1990.
\bibitem{Swe}
N.H. Sweilam, A.M. Nagy, Adel A. El Sayed, Second kind shifted Chebyshev polynomials for solving space fractional order diffusion equation, Chaos Solition Fract, 73, 141-147, 2015.
\bibitem{WitSl}
R. Witula, D. Slota, On modified Chebyshev polynomials, J. Math. Anal. Appl., 324(1), 321-343, 2006.
\bibitem{WitSl2}
R. Witula, D. Slota, Cardano's formula, square roots, Chebyshev polynomials and radicals, J. Math. Anal. Appl., 363(1), 639-647, 2010.
\bibitem{Yama}
M. Yamagishi, A note on Chebyshev polynomials, cyclotomic polynomials and twin primes, J. Number Theor., 133, 2455-2463, 2013.
\end{thebibliography}
\end{document}